\documentclass[10pt,a4paper]{amsart}
\usepackage{amsfonts}
\usepackage{amssymb,amsmath,amsthm,stmaryrd}
\usepackage{multirow,bigdelim}
\usepackage[american]{babel}

\title{An asymptotically tight bound for \linebreak the Davenport constant}
\author{Benjamin Girard}
\keywords{Additive combinatorics, zero-sum sequences, Davenport constant, finite Abelian groups}
\subjclass[2010]{05E15, 11B30, 11B75, 11A25, 20D60, 20K01}
\thanks{Sorbonne Universit\'e, Universit\'e Paris Diderot, CNRS, Institut de Math\'ematiques de Jussieu - Paris Rive Gauche, IMJ-PRG, F-75005, Paris, France, email: \texttt{benjamin.girard@imj-prg.fr}}

\newtheorem{theorem}{Theorem}

\newtheorem{conjecture}{Conjecture}

\def\le{\leqslant}
\def\ge{\geqslant}

\begin{document}
$ $
\vspace{-1.5cm}
\maketitle \setcounter{page}{1} 

\vspace{-0.8cm}
\begin{abstract} 
We prove that for every integer $r \ge 1$ the Davenport constant $\mathsf{D}(C^r_n)$ is asymptotic to $rn$ when $n$ tends to infinity.
An extension of this theorem is also provided.
\end{abstract}

\vspace{0.4cm}
For every integer $n \ge 1$, let $C_n$ be the cyclic group of order $n$. 
It is well known that every non-trivial finite Abelian group $G$ can be uniquely decomposed as a direct product of cyclic groups 
$C_{n_1} \oplus \cdots \oplus C_{n_r}$ such that $1 < n_1 \mid \cdots \mid n_r \in \mathbb{N}$. 
The integers $r$ and $n_r$ appearing in this decomposition are respectively called the rank and the exponent of $G$.
The latter is denoted by $\exp(G)$.
For the trivial group, the rank is $0$ and the exponent is $1$.
For every integer $1 \le d \mid \exp(G)$, we denote by $G_d$ the subgroup of $G$ consisting of all elements of order dividing $d$. 

\medskip
Any finite sequence $S$ of $\ell$ elements of $G$  will be called a sequence over $G$ of length $|S|=\ell$. 
Also, we denote by $\sigma(S)$ the sum of all elements in $S$. 
The sequence $S$ will be referred to as a zero-sum sequence whenever $\sigma(S)=0$.

\medskip
By $\mathsf{D}(G)$ we denote the smallest integer $t \ge 1$ such that every sequence $S$ over $G$ of length $|S| \ge t$ contains a non-empty zero-sum subsequence.
This number, which is called the Davenport constant, drew over the last fifty years an ever growing interest, most notably in additive combinatorics and algebraic number theory. A detailed account on the many aspects of this invariant can be found in \cite{CDG16,GaoGero06,GeroRuzsa09,GeroKoch05,Narkie04}. 

\medskip
To name but one striking feature, let us recall the Davenport constant has the following arithmetical interpretation.
Given the ring of integers $\mathcal{O}_\textbf{K}$ of some number field $\textbf{K}$ with ideal class group $G$, the maximum number of prime ideals in the decomposition of an irreducible element of $\mathcal{O}_\textbf{K}$ is $\mathsf{D}(G)$ \cite{Rogers63}. 
The importance of this fact is best highlighted by the following generalization of the prime number theorem \cite[Theorem 9.15]{Narkie04},
stating that the number $F(x)$ of pairwise non-associated irreducible elements in $\mathcal{O}_\textbf{K}$ whose norms do not exceed $x$ in absolute value satisfies,
$$F(x) \underset{x \rightarrow +\infty}{\sim} C\frac{x}{\log x}(\log\log x)^{\mathsf{D}(G)-1},$$
with a suitable constant $C > 0$ depending solely on $G$ (see \cite[Chapter 9.1]{GeroKoch05} and \cite[Theorem 1.1]{Kac09} for sharper and more general results).  

\medskip
We are thus naturally led to the problem of determining the exact value of $\mathsf{D}(G)$. 
The best explicit bounds known so far are
\begin{equation}
\label{bornes Dav}
\displaystyle\sum^r_{i=1} (n_i-1)+1  \le  \mathsf{D}(G) \le n_r\left(1 + \log \frac{|G|}{n_r}\right).
\end{equation}
The lower bound follows easily from the fact that if $(e_1,\dots,e_r)$ is a basis of $G$ such that $\text{ord}(e_i)=n_i$ for all $i \in \llbracket 1,r \rrbracket$, the sequence $S$ consisting of $n_i-1$ copies of $e_i$ for each $i \in \llbracket 1,r \rrbracket$ contains no non-empty zero-sum subsequence.
The upper bound first appeared in \cite[Theorem 7.1]{EmdeBoas69'} and was rediscovered in \cite[Theorem 1]{Meshulam90}. 
See also \cite[Theorem 1.1]{Alford94} for a reformulation of the proof's original argument as well as an application of the Davenport constant to the study of Carmichael numbers.  

\medskip
$\mathsf{D}(G)$ has been proved to match the lower bound in (\ref{bornes Dav}) when $G$ is either a $p$-group \cite{Olso69a} or has rank at most $2$ \cite[Corollary 1.1]{Olso69b}. Even though there are infinitely many finite Abelian groups whose Davenport constant is known to exceed this lower bound \cite{EmdeBoas69',GeroLieb12,GeroSchneider92,Mazur92}, none of the ones identified so far either have rank $3$ or the form $C^r_n$. 
Since the late sixties, these two types of groups have been conjectured to have a Davenport constant matching the lower bound in (\ref{bornes Dav}). 
This open problem was first raised in \cite[pages 13 and 29]{EmdeBoas69'} and can be found formally stated as a conjecture in \cite[Conjecture 3.5]{GaoGero06}. 
See also \cite[Conjecture A.5]{Alon84} and \cite[Theorem 6.6]{GaoGero03} for connections with graph theory and covering problems. 

\begin{conjecture}
\label{conject}
For all integers $n,r \ge 1$,
$$\mathsf{D}(C^r_n) = r(n-1)+1.$$
\end{conjecture}

Besides the already mentioned results settling Conjecture \ref{conject} for all $r$ when $n$ is a prime power and for all $n$ when $r \le 2$, note that 
$\mathsf{D}(C^3_n)$ is known only when $n=2p^\alpha$, with $p$ prime and $\alpha \ge 1$ \cite[Corollary 4.3]{EmdeBoas69}, or $n=2^\alpha 3$ with $\alpha \ge 2$ \cite[Corollary 1.5]{EmdeBoas69'}, and satisfies Conjecture \ref{conject} in both cases. 
To the best of our knowledge, the exact value of $\mathsf{D}(C^r_n)$ is currently unknown for all pairs $(n,r)$ such that $n$ is not a prime power and $r \ge 4$. 
In all those remaining cases, the bounds in (\ref{bornes Dav}) translate into
\begin{equation}
\label{bornes Dav C^r_n}
r(n-1)+1 \le \mathsf{D}(C^r_n) \le n\left(1 + (r-1)\log n\right),
\end{equation}
which leaves a substantial gap to be bridged.
Conjecture \ref{conject} thus remains wide open. 

\medskip
The aim of the present note is to clarify the behavior of $\mathsf{D}(C^r_n)$ for any fixed $r \ge 1$ when $n$ goes to infinity. 
Our main theorem proves Conjecture $1$ in the following asymptotic sense. 

\begin{theorem}
\label{equiv Dav}
For every integer $r \ge 1$,
$$\mathsf{D}(C^r_n) \underset{n \rightarrow +\infty}{\sim} rn.$$
\end{theorem}

The proof of Theorem \ref{equiv Dav} relies on a new upper bound for $\mathsf{D}(C^r_n)$, turning out to be a lot sharper than the 
one in (\ref{bornes Dav C^r_n}) for large values of $n$. So as to state it properly, we now make the following definition. 
For every integer $n \ge 1$, we denote by $P(n)$ the greatest prime power dividing $n$, with the convention $P(1)=1$.

\begin{theorem}
\label{main result bis}
For every integer $r \ge 1$, there exists a constant $d_r \ge 0$ such that for every integer $n \ge 1$,
$$\mathsf{D}(C^r_n) \le r\left(n-1\right) + 1 + d_r\left(\frac{n}{P(n)}-1\right).$$
\end{theorem}

The relevance of this bound to the study of the Davenport constant is due to the fact that the arithmetic function $P(n)$ tends to infinity when $n$ does so.
Indeed, if we denote by $\mathcal{P}$ the set of prime numbers and let $(a_n)_{n \ge 1}$ be the sequence defined for every integer $n \ge 1$ by 
$$a_n = \prod_{p \in \mathcal{P}} p^{\left\lfloor \frac{\log n}{\log p} \right\rfloor},$$
we easily notice that, for every integer $N \ge 1$, one has $P(n) > N$ as soon as $n > a_N$. 

\medskip
Now, since $P(n)$ tends to infinity when $n$ does so, 
Theorem \ref{main result bis} allows us to deduce that, 
for every integer $r \ge 1$, the gap between the Davenport constant and its conjectural value  
$$\mathsf{D}(C^r_n) - \left(r\left(n-1\right) + 1\right)$$
is actually $o(n)$.
This theorem will be obtained via the inductive method, which involves another key combinatorial invariant we now proceed to define.

\medskip
By $\eta(G)$ we denote the smallest integer $t \ge 1$ such that every sequence $S$ over $G$ of length $|S| \ge t$ contains a non-empty zero-sum subsequence $S' \mid S$ with $|S'| \le \exp(G)$. 
It is readily seen that $\mathsf{D}(G) \le \eta(G)$ for every finite Abelian group $G$. 

\medskip
A natural construction shows that, for all integers $n,r\ge 1$, one has 
\begin{equation}
\label{min eta}
(2^r-1)(n-1)+1 \le \eta(C^r_n).
\end{equation}
Indeed, if $(e_1,\dots,e_r)$ is a basis of $C^r_n$, it is easily checked that the sequence $S$ consisting 
of $n-1$ copies of $\sum_{i \in I} e_i$ for each non-empty subset $I \subseteq \llbracket 1,r \rrbracket$ contains no non-empty zero-sum subsequence of length at most $n$.

\medskip
The exact value of $\eta(C^r_n)$ is known to match the lower bound in (\ref{min eta}) for all $n$ when $r \le 2$ \cite[Theorem 5.8.3]{GeroKoch05}, and for all $r$ when $n=2^\alpha$, with $\alpha \ge 1$ \cite[Satz 1]{Harborth73}. 
Besides these two results, $\eta(C^r_n)$ is currently known only when $r=3$ and $n=3^\alpha5^\beta$, with $\alpha,\beta \ge 0$ \cite[Theorem 1.7]{GaoSchmid06}, in which case $\eta(C^3_n)=8n-7$, or $n=2^\alpha 3$, with $\alpha \ge 1$ \cite[Theorem 1.8]{GaoSchmid06}, in which case $\eta(C^3_n)=7n-6$. 
When $n=3$, note that the problem of finding $\eta(C^r_3)$ is closely related to the well-known cap-set problem, and that for $r \ge 4$, the only known values so far are $\eta(C^4_3)=39$ \cite{Pellegrino71}, $\eta(C^5_3)=89$ \cite{EdelFerret02} and $\eta(C^6_3)=223$ \cite{Potechin08}.
For more details on this fascinating topic, see \cite{EdelGero06,EllenGijs16} and the references contained therein.   

\medskip
In another direction, Alon and Dubiner showed \cite{Alodub95} that when $r$ is fixed, $\eta(C^r_n)$ grows linearly in the exponent $n$.
More precisely, they proved that for every integer $r \ge 1$, there exists a constant $c_r > 0$ such that for every integer $n \ge 1$, 
\begin{equation}
\label{maj eta}
\eta\left(C^r_n\right) \le c_r(n-1)+1.
\end{equation}
From now on, we will identify $c_r$ with its smallest possible value in this theorem. 

\medskip
On the one hand, it follows from (\ref{min eta}) that $c_r \ge 2^r-1$, for all $r \ge 1$. 
Since, as already mentioned, $\eta(C_n)=n$ and $\eta(C^2_n)=3n-2$ for all $n \ge 1$, it is possible to choose $c_1=1$ and $c_2=3$, with equality in (\ref{maj eta}). 

\medskip
On the other hand, the method used in \cite{Alodub95} yields $c_r \le \left( cr\log r\right)^r$, where $c > 0$ is an  absolute constant, and it is conjectured in \cite{Alodub95} that there actually is an absolute constant $d > 0$ such that $c_r \le d^r$ for all $r \ge 1$. 

\medskip
We can now state and prove our first technical result, which is the following.

\begin{theorem}
\label{main result}
For all integers $n,r \ge 1$,
$$\mathsf{D}(C^r_n) \le r\left(n-1\right) + 1 + (c_r-r)\left(\frac{n}{P(n)}-1\right).$$
\end{theorem}

\begin{proof}[Proof of Theorem \ref{main result}]
We set $G=C^r_n$ and denote by $H=G_{P(n)}$ the largest Sylow subgroup of $G$. 
Since $H \simeq C^r_{P(n)}$ is a $p$-group, it follows from \cite{Olso69a} that
$$\mathsf{D}(H) = r(P(n)-1)+1.$$
In addition, since the quotient group $G \slash H \simeq C^r_{n \slash P(n)}$ has exponent $n \slash P(n)$ and rank at most $r$, it follows from (\ref{maj eta}) that
$$\eta(G \slash H) \le c_r\left(\frac{n}{P(n)}-1\right)+1.$$
Now, from any sequence $S$ over $G$ such that
$$|S| \ge \exp(G \slash H) \left(\mathsf{D}(H) -1 \right) + \eta(G \slash H),$$
one can sequentially extract at least $d = \mathsf{D}(H)$ disjoint non-empty subsequences $S'_1,\dots,S'_{d} \mid S$ such that $\sigma(S'_i) \in H$ and $|S'_i| \le \exp(G \slash H)$ for every $i \in \llbracket 1,d \rrbracket$ (see for instance \cite[Lemma 5.7.10]{GeroKoch05}).
Since $T=\prod^d_{i=1} \sigma(S'_i)$ is a sequence over $H$ of length $|T|=\mathsf{D}(H)$, there exists a non-empty subset $I \subseteq \llbracket 1,d \rrbracket$ such that $T'=\prod_{i \in I} \sigma(S'_i)$ is a zero-sum subsequence of $T$. Then, $S'=\prod_{i \in I} S'_i$ is a non-empty zero-sum subsequence of $S$.

\medskip
Therefore, we have
\begin{eqnarray*}
\mathsf{D}(G) & \le & \exp(G \slash H) \left( \mathsf{D}(H) -1 \right) + \eta(G \slash H) \\
     & \le &  \frac{n}{P(n)}\left(r(P(n)-1)\right) + c_r\left(\frac{n}{P(n)}-1\right) +1 \\
     & = & r(n-1) +1 + (c_r-r)\left(\frac{n}{P(n)}-1\right),
\end{eqnarray*}
which completes the proof.
\end{proof}

Note that Theorem \ref{main result} is sharp for all $n$ when $r=1$ and for all $r$ when $n$ is a prime power. 
Also, Theorems \ref{equiv Dav} and \ref{main result bis} are now direct corollaries of Theorem \ref{main result}.

\begin{proof}[Proof of Theorem \ref{main result bis}]
The result follows from Theorem \ref{main result} by setting $d_r=c_r-r$. 
\end{proof}

\begin{proof}[Proof of Theorem \ref{equiv Dav}]
Since $P(n)$ tends to infinity when $n$ does so, the desired result follows easily from (\ref{bornes Dav C^r_n}) and Theorem \ref{main result bis}.
\end{proof}

To conclude this paper, we would like to offer a possibly useful extension of our theorems to the following wider framework.  
Given any finite Abelian group $L$ and any integer $r \ge 1$, we consider the groups defined by
$L^r_n = L \oplus C^r_n$, where $n \ge 1$ is any integer such that $\exp(L)  \mid n$.
Note that if $L$ is the trivial group, then $L^r_n \simeq C^r_n$ whose Davenport constant is already covered by Theorems \ref{equiv Dav}-\ref{main result}.

\medskip
Our aim in this more general context is to prove that, for every finite Abelian group $L$ and every integer $r \ge 1$, 
$\mathsf{D}(L^r_n)$ behaves asymptotically in the same way it would if $L$ were trivial.
To do so, we establish the following extension of Theorem~\ref{main result}.

\begin{theorem}
\label{main result general form}
Let $L\simeq C_{n_1} \oplus \cdots \oplus C_{n_{\ell}}$, with $1 < n_1 \mid \cdots \mid n_{\ell} \in \mathbb{N}$, be a finite Abelian group. 
For every integer $n \ge 1$ such that $\exp(L) \mid n$ and every integer $r \ge 1$,
$$\mathsf{D}(L^r_n) \le r\left(n-1\right) + 1 + (c_{\ell+r}-r)\left(\frac{n}{P(n)}-1\right) + \frac{n}{P(n)} \displaystyle\sum^{\ell}_{i=1} (\gcd(n_i,P(n))-1).$$
\end{theorem}

\begin{proof}[Proof of Theorem \ref{main result general form}]
We set $G=L^r_n$ and $H=G_{P(n)}$. 
On the one hand, since $H \simeq C_{n'_1} \oplus \cdots \oplus C_{n'_{\ell}} \oplus C^r_{P(n)}$, with $n'_i = \gcd(n_i,P(n)) \mid n_i$ for all $i \in \llbracket 1,\ell \rrbracket$ and $1 \le n'_1 \mid \cdots \mid n'_{\ell} \mid P(n)$, is a $p$-group, it follows from \cite{Olso69a} that
$$\mathsf{D}(H) = \displaystyle\sum^{\ell}_{i=1} (n'_i-1) + r(P(n)-1)+1.$$
On the other hand, since the quotient group $G \slash H$ has exponent $n \slash P(n)$ and rank at most $\ell+r$, it follows from (\ref{maj eta}) that
$$\eta(G \slash H) \le \eta\left(C^{\ell+r}_{\frac{n}{P(n)}}\right) \le c_{\ell + r}\left(\frac{n}{P(n)}-1\right)+1.$$
Therefore, the same argument we used in our proof of Theorem \ref{main result} yields
\begin{eqnarray*}
\mathsf{D}(G) & \le & \exp(G \slash H) \left( \mathsf{D}(H) -1 \right) + \eta(G \slash H) \\
     & \le &  \frac{n}{P(n)}\left(\displaystyle\sum^{\ell}_{i=1} (n'_i-1)+r(P(n)-1)\right) + c_{\ell + r}\left(\frac{n}{P(n)}-1\right) +1 \\
     & = & r\left(n-1\right) + 1 + (c_{\ell + r}-r)\left(\frac{n}{P(n)}-1\right) + \frac{n}{P(n)} \displaystyle\sum^{\ell}_{i=1} (n'_i-1),
\end{eqnarray*}
which is the desired upper bound.
\end{proof}

Theorem \ref{main result general form} now easily implies the following generalization of Theorem \ref{equiv Dav}.

\begin{theorem}
\label{general equiv Dav}
For every finite Abelian group $L$ and every integer $r \ge 1$,
$$\mathsf{D}(L^r_n) \underset{\substack{n \rightarrow +\infty \\ \exp(L) \mid n}}{\sim} rn.$$
\end{theorem}

\begin{proof}[Proof of Theorem \ref{general equiv Dav}] 
We write $L \simeq C_{n_1} \oplus \cdots \oplus C_{n_{\ell}}$, with $1 < n_1 \mid \cdots \mid n_{\ell} \in \mathbb{N}$.
For every integer $n \ge 1$ such that $\exp(L) \mid n$, one has $\gcd(n_i,P(n)) \le n_i$ for all $i \in \llbracket 1,\ell \rrbracket$. 
Since $P(n)$ tends to infinity when $n$ does so, the result follows easily from (\ref{bornes Dav}) and Theorem \ref{main result general form}.
\end{proof}

 \section*{Acknowledgements}
 The author is grateful to W.A. Schmid for his careful reading of the manuscript in an earlier version.

\end{document}